\documentclass[11pt,a4paper]{article}
  \def\SvZfontcode{8} 
  \def\SvZslantedGreekCapitals{1}
  \usepackage{amsmath,amsthm,amscd}

\usepackage[T1]{fontenc}

\def\SvZrequireslantedRedef{0}
\ifcase\SvZfontcode
\usepackage{bm}
\usepackage{amssymb}
\def\SvZrequireslantedRedef{1}
\or
\usepackage{fouriernc}
\usepackage{bm}
\def\SvZrequireslantedRedef{1}
\or
\usepackage{fourier}
\usepackage{bm}
\def\SvZrequireslantedRedef{1}
\or
\usepackage{amssymb}
\if\SvZslantedGreekCapitals 1
\usepackage[slantedGreek]{mathptmx}
\else
\usepackage{mathptmx}
\fi
\DeclareMathAlphabet{\bm}{OT1}{ptm}{b}{it} 
\or
\usepackage{kmath,kerkis}
\usepackage{bm}
\def\SvZrequireslantedRedef{1}
\or
\if\SvZslantedGreekCapitals 1
\usepackage[sc,slantedGreek]{mathpazo}
\else
\usepackage[sc]{mathpazo}
\fi
\usepackage{bm}
\or
\usepackage[adobe-utopia]{mathdesign}
\usepackage{bm}
\or
\usepackage[urw-garamond]{mathdesign}
\usepackage{bm}
\or
\usepackage[bitstream-charter]{mathdesign}
\usepackage{bm}
\or
\usepackage{lmodern}
\usepackage{bm}
\usepackage{amssymb}
\def\SvZrequireslantedRedef{1}
\or
\usepackage{arev}
\fi

\if\SvZrequireslantedRedef 1
\if\SvZslantedGreekCapitals 1
\renewcommand{\Gamma}{\varGamma}
\renewcommand{\Delta}{\varDelta}
\renewcommand{\Theta}{\varTheta}
\renewcommand{\Lambda}{\varLambda}
\renewcommand{\Xi}{\varXi}
\renewcommand{\Pi}{\varPi}
\renewcommand{\Sigma}{\varSigma}
\renewcommand{\Upsilon}{\varUpsilon}
\renewcommand{\Phi}{\varPhi}
\renewcommand{\Psi}{\varPsi}
\renewcommand{\Omega}{\varOmega}
\fi
\fi

\renewcommand{\phi}{\varphi}

\reversemarginpar

\usepackage{url}
\usepackage{listings}
\usepackage{kbordermatrix}
\usepackage{booktabs}
\usepackage{graphicx}
\usepackage{color}

\ifx\SvZinPresentation\undefined
\usepackage{float}
\usepackage{longtable}
\usepackage[margin=10pt,labelfont=bf,textfont=it]{caption}

\usepackage{hyperref}
\hypersetup{
    colorlinks,%
    citecolor=black,%
    filecolor=black,%
    linkcolor=black,%
    urlcolor=black
}

\fi

\newcommand{\mathds}{\mathbb}

\DeclareMathOperator{\rank}{rk}

\newcommand{\N}{\mathds{N}}
\newcommand{\field}{\mathds{F}}

\DeclareMathOperator{\GF}{GF}

\newcommand{\ignore}[1]{}

\DeclareMathOperator{\closure}{cl}

\DeclareMathOperator{\PG}{PG} 

\let\Oldsetminus\setminus
\renewcommand{\setminus}{\ensuremath{-}}

\newcommand{\delete}{\ensuremath{\!\Oldsetminus}}
\newcommand{\contract}{\ensuremath{\!/}}

\ifx\SvZinPresentation\undefined
\newtheorem{theorem}{Theorem}[section]
\else
\newtheorem{theorem}{Theorem}
\fi

\newtheorem{lemma}[theorem]{Lemma}

\newtheorem{claim}{Claim}[theorem] 
\newtheorem{conjecture}[theorem]{Conjecture}

\newenvironment{claimenv}{\list{}{\rightmargin0pt\leftmargin10pt\topsep0pt}\item[]}{\endlist}

\newenvironment{subproof}{\begin{claimenv}\begin{proof}}{\end{proof}\end{claimenv}}
\theoremstyle{definition}
\newtheorem{definition}[theorem]{Definition}

\newtheorem{problem}[theorem]{Problem}

\DeclareMathOperator{\localconn}{\sqcap}
\newcommand{\colocalconn}{\localconn^*}

\newcommand{\coclosure}{\closure^*}

\usepackage[numbers]{natbib}

\newcommand{\smin}{-}

\begin{document}
\title{Intertwining connectivities in representable matroids\footnote{Research supported by the NWO (The Netherlands Organization for Scientific Research) free competition project ``Matroid Structure -- for Efficiency'' led by Bert Gerards, and by the National Science Foundation, Grant No. 1161650.}}
\author{Tony Huynh\footnote{Department of Mathematics, Simon Fraser University, Burnaby, B.C., Canada. Email: \url{tony.bourbaki@gmail.com}} \and Stefan H. M. van Zwam \footnote{Department of Mathematics, Princeton University, Princeton, NJ, United States. Email: \url{svanzwam@math.princeton.edu}}}

\maketitle

\abstract{
  Let $M$ be a representable matroid, and $Q, R, S, T$ subsets of the ground set such that the smallest separation that separates $Q$ from $R$ has order $k$, and the smallest separation that separates $S$ from $T$ has order $l$. We prove that, if $M$ is sufficiently large, then there is an element $e$ such that in one of $M\delete e$ and $M\contract e$ both connectivities are preserved. 
  
  For matroids representable over a finite field we prove a stronger result: we show that we can remove $e$ such that both a connectivity and a minor of $M$ are preserved.
}

\section{Introduction} 
\label{sec:introduction}
For a matroid $M$ on ground set $E$ we define, as usual, the \emph{connectivity function} $\lambda_M$ by $\lambda_M(X) := \rank_M(X) +\rank_M(E\smin X) - \rank(M)$. For disjoint sets $S,T\subseteq E$, the \emph{connectivity between $S$ and $T$} is
\begin{align}
    \kappa_M(S,T) := \min\{\lambda_M(X) : S \subseteq X \subseteq E\smin T\}.
\end{align}

Geelen, in private communication, conjectured the following.

\begin{conjecture}\label{con:mainres}
    There exists a function $c: \N^2 \to \N$ with the following property. Let $M$ be a matroid, and let $Q, R, S, T \subseteq E(M)$ be sets of elements such that $Q\cap R = S\cap T = \emptyset$. Let $k := \kappa_M(Q,R)$ and $l := \kappa_M(S,T)$. If $|E(M)\smin (Q\cup R\cup S \cup T)| \geq c(k,l)$, then there exists an element $e\in E(M)\smin (Q\cup R\cup S \cup T)$ such that one of the following holds:
    \begin{enumerate}
        \item $\kappa_{M\delete e}(Q, R) = k$ and $\kappa_{M\delete e}(S, T) = l$;
        \item $\kappa_{M\contract e}(Q, R) = k$ and $\kappa_{M\contract e}(S, T) = l$.
    \end{enumerate}
\end{conjecture}    

In other words, for fixed $Q, R, S, T$, there is a finite number of minor-minimal matroids with the prescribed connectivities. This formulation is reminiscent of the definition of an \emph{intertwine}, which is a minor-minimal matroid containing two prescribed \emph{minors}. For that reason we speak of the intertwining of connectivities.

For graphs the result follows readily from Robertson and Seymour's Graph Minors Theorem \cite{RSXX}. In this paper we prove the conjecture for all representable matroids.

For matroids representable over a finite field we prove a stronger result:
\begin{theorem}\label{thm:intertwinewithminor}
        There exists a function $c:\N^3 \to \N$ with the following property.
    Let $q$ be a prime power, let $M$ be a $\GF(q)$-representable matroid, let $N$ be a minor of $M$, let $S,T \subseteq E(M)$ be disjoint, and let $k := \kappa_M(S, T)$. If $|E(M)\smin (S\cup T\cup E(N))| > c(q,|E(N)|,k)$, then there exists an element $e \in E(M)\smin (S\cup T\cup E(N))$ such that at least one of the following holds:
    \begin{enumerate}
        \item $\kappa_{M\delete e}(S, T) = k$ and $N$ is a minor of $M\delete e$;
        \item $\kappa_{M\contract e}(S, T) = k$ and $N$ is a minor of $M\contract e$.
    \end{enumerate}
\end{theorem}

By repeated use of this theorem, it is possible to bound the size of an intertwine of any fixed number of connectivities. This gives a (highly unsatisfying) answer to the following problem:
\begin{problem}\label{prob:gammoid}
    Let $M = (S, \mathcal{I})$ be a matroid that is a gammoid. Give an upper bound, in terms of $|S|$, on the size of the graph needed to represent $M$ as a gammoid.
\end{problem}
Good upper bounds can potentially be useful in the study of parametrized complexity (c.f. \cite{Mar09}).

Our proof technique for Theorem \ref{thm:intertwinewithminor} has been used previously in, for instance, \cite{GGW02,GHW05,Kral07}. For graphs it dates back at least to the work of Robertson and Seymour on graph minors (cf. \cite{RSXXI}). In fact, Theorem \ref{thm:intertwinewithminor} is a generalization of \cite[Theorem 1.1]{GHW05} and \cite[Theorem 13.3]{TruIII}.

Theorem \ref{thm:intertwinewithminor} becomes false when the dependence on $q$ is removed. A counterexample is readily obtained from a construction of arbitrarily long blocking sequences in \cite[Proposition 6.1]{GHW05}. It follows that different techniques are needed to prove Conjecture \ref{con:mainres}.  

Our proof of Conjecture \ref{con:mainres} for representable matroids uses a different approach, based on a suggestion by Geelen (private communication). Unfortunately, the proof uses a property of representable matroids that does not hold for general matroids.

The paper is organized as follows. In Section \ref{sec:conn} we fix some terminology and state some easy lemmas. Section \ref{sec:menger} contains results related to Tutte's Linking Theorem. The main result in that section shows that, if Conjecture \ref{con:mainres} is false, there exist matroids with arbitrarily long sequences of nested separations. In Section \ref{sec:finitefield} we prove Theorem \ref{thm:intertwinewithminor}, and in Section \ref{sec:mainres} we prove Conjecture \ref{con:mainres} for all representable matroids.

\section{Preliminaries}\label{sec:conn}
We will use the following elementary observation (cf. \cite{OSW10a,GW10}):
\begin{lemma}\label{lem:closurecomplement}
    Let $M$ be a matroid and let $(A,\{e\}, B)$ be a partition of $E(M)$. Then $e \in \closure_M(A)$ if and only if $e\not\in\coclosure_M(B)$.
\end{lemma}

It is well-known that the connectivity function is submodular:
\begin{lemma}\label{lem:submod}
    Let $M$ be a matroid, and let $X,Y\subseteq E(M)$. Then
    \begin{align*}
        \lambda_M(X) + \lambda_M(Y) \geq \lambda_M(X\cap Y) + \lambda_M(X\cup Y).
    \end{align*}
\end{lemma}

The following lemmas are easily verified:
\begin{lemma}\label{lem:lambdamonotone}
    Let $M$ be a matroid, let $X\subseteq E(M)$, and let $N$ be a minor of $M$ with $X \subseteq E(N)$. Then $\lambda_N(X) \leq \lambda_M(X)$.
\end{lemma}

\begin{lemma}\label{lem:kappamonotone}
    Let $M$ be a matroid, let $S,T$ be disjoint subsets of $E(M)$, and let $N$ be a minor of $M$ with $S\cup T \subseteq E(N)$. Then $\kappa_N(S,T) \leq \kappa_M(S,T)$.
\end{lemma}

We introduce some terminology.
\begin{definition}\label{def:convenient}
    Let $M$ be a matroid and let $S,T$ be disjoint subsets of $E(M)$. A partition $(A,B)$ of $E(M)$ is \emph{$S-T$-separating of order $k+1$} if $S\subseteq A$, $T\subseteq B$, and $\lambda_M(A) = k$. If $B$ is implicit, we also say that $A$ is \emph{$S-T$-separating} of order $k+1$. 
\end{definition}
If, moreover, $|A|, |B| \geq k+1$ then $(A,B)$ is an (exact) \emph{$(k+1)$-separation} of $M$. Sometimes we will be sloppy and say that $(A,B)$ is $S-T$ separating if $S\subseteq B$ and $T\subseteq A$.

\begin{lemma}\label{lem:modular}
    Let $M$ be a matroid, let $S,T\subseteq E(M)$ be disjoint subsets, and let $k := \kappa_M(S,T)$. If $(A_1, B_1)$ and $(A_2, B_2)$ are $S-T$-separating with $\lambda_M(A_1) = \lambda_M(A_2) = k$, then $(A_1\cap A_2, B_1\cup B_2)$ is $S-T$-separating of order $k+1$.
\end{lemma}

\begin{proof}
    Clearly, $(A_1\cap A_2, B_1\cup B_2)$ and $(A_1 \cup A_2, B_1 \cap B_2)$ are $S-T$-separating. Since $\kappa_M(S,T) = k$, we must have $\lambda_M(A_1\cap A_2) \geq k$ and $\lambda_M(A_1\cup A_2) \geq k$. It follows from Lemma \ref{lem:submod} that equality must hold.
\end{proof}

Finally, we will frequently use the following well-known result and its dual.
\begin{lemma}\label{lem:contractclosure}
    Let $M$ be a matroid, let $S,T \subseteq E(M)$ be disjoint subsets, let $k := \kappa_M(S,T)$, and let $e\in E(M)\smin (S\cup T)$. A partition $(A,B)$ of $E(M)\smin e$ is $S-T$-separating of order $k$ in $M\contract e$ if and only if $(A\cup e, B)$ is $S-T$-separating of order $k+1$ in $M$ with $e\in\closure_M(A)\cap\closure_M(B)$.
\end{lemma}

\section{Tutte's Linking Theorem}\label{sec:menger}

In \cite{Tut65a}, Tutte proved the following result, which can be seen to be a generalization of Menger's theorem to matroids (see \cite[Section 8.5]{ox2}):

\begin{theorem}\label{thm:tuttelinkminmax}
    Let $M$ be a matroid and let $S, T$ be disjoint subsets of $E(M)$. Then
    \begin{align}
        \kappa_M(S,T) = \max \{ \lambda_N(S) : N \textrm{ minor of } M \textrm{\ such that\ } E(N) = S\cup T\}.
    \end{align}
\end{theorem}

Equivalently,

\begin{theorem}\label{thm:tuttelink}
    Let $M$ be a matroid and let $S,T$ be disjoint subsets of $E(M)$. For each $e \in E(M)\smin (S \cup T)$, at least one of the following holds:
    \begin{enumerate}
        \item $\kappa_{M\delete e}(S,T) = \kappa_M(S,T)$, or
        \item $\kappa_{M\contract e}(S,T) = \kappa_M(S,T)$.
    \end{enumerate}
\end{theorem}

\begin{definition}\label{def:Mengerfragile}
    Let $M$ be a matroid, let $S,T$ be disjoint subsets of $E(M)$, and let $e \in E(M)\smin (S\cup T)$. 
    \begin{enumerate}
        \item If $\kappa_{M\delete e}(S,T) = \kappa_M(S,T)$ then we say $e$ is \emph{deletable with respect to $(S,T)$}.
        \item If $\kappa_{M\contract e}(S,T) = \kappa_M(S,T)$ then we say $e$ is \emph{contractible with respect to $(S,T)$}.
        \item If $e$ is both deletable and contractible then we say $e$ is \emph{flexible with respect to $(S,T)$}.
    \end{enumerate}
\end{definition}
We may omit the phrase ``with respect to $(S,T)$'' if it can be deduced from the context. We will mainly be concerned with non-flexible elements. The following theorem is the main result of this section:

\begin{theorem}\label{thm:mengerfragile}
    Let $M$ be a matroid, let $S,T$ be disjoint subsets of $E(M)$, let $k := \kappa_M(S,T)$, and let $F\subseteq E(M)\smin (S\cup T)$ be a set of non-flexible elements. There exist an ordering $(f_1, f_2, \ldots, f_t)$ of $F$ and a sequence $(A_1, A_2, \ldots, A_t)$ of subsets of $E(M)$, such that
    \begin{enumerate}
        \item $A_i$ is $S-T$-separating of order $k+1$ for each $i \in \{1,\ldots,t\}$;
        \item $A_{i} \subseteq A_{i+1}$ for each $i\in \{1,\ldots, t-1\}$;
        \item $A_i \cap F = \{f_1, \dots, f_i\}$ for each $i \in \{1,\ldots,t\}$;
        \item $f_i \in \closure_M(A_i\smin f_i) \cap \closure_M(E(M)\smin A_i)$ or $f_i \in \coclosure_M(A_i\smin f_i) \cap \coclosure_M(E(M)\smin A_i)$.
    \end{enumerate}
\end{theorem}

We will need two lemmas to prove this theorem.

\begin{lemma}\label{lem:minimalsep}
    Let $M$ be a matroid, let $S,T$ be disjoint subsets of $E(M)$, let $k := \kappa_M(S,T)$, and let $e \in E(M)\smin (S\cup T)$ be non-contractible. If $(A,B)$ is an $S-T$-separating partition of order $k+1$ such that $e\in A$ and $|A|$ is minimum, then $e\in\closure_M(A\smin e)\cap\closure_M(B)$.
\end{lemma}

\begin{proof}
    Suppose not. By Lemma \ref{lem:contractclosure}, there is an $S-T$-separating partition $(A', B')$ of order $k+1$ such that $e\in A'$ and $e \in \closure_M(A'\smin e) \cap \closure_M(B')$. By Lemma \ref{lem:modular}, $A\cap A'$ is $S-T$-separating of order $k+1$. By minimality of $A$, it then follows that $A \subseteq A'$, and therefore $B \supseteq B'$. But then $e\in \closure_M(B)$. By Lemma \ref{lem:closurecomplement}, then, $e\not\in\coclosure_M(A\smin e)$. If also $e\not\in\closure_M(A\smin e)$ then $\lambda_M(A\smin e) = k-1$, contradicting $\kappa_M(S,T)=k$. The result follows.
\end{proof}

\begin{lemma}\label{lem:growS}
    Let $M$ be a matroid, let $S,T$ be disjoint subsets of $E(M)$, let $k := \kappa_M(S,T)$, and let $U$ be an $S-T$-separating set of order $k+1$. If $e \in E(M)\smin (T\cup U)$ is non-contractible with respect to $(S,T)$, then $e$ is non-contractible with respect to $(U,T)$.
\end{lemma}

\begin{proof}
    First, observe that $\kappa_M(U,T) = k$. If the lemma is false, then there is an $S-T$-separating partition $(A,B)$ of order $k$ in $M\contract e$, yet $\kappa_{M\contract e}(U,T) = k$. In particular, $\lambda_{M\contract e}(U) = k$. By submodularity,
    \begin{align}
        2k-1 = \lambda_{M\contract e}(A) + \lambda_{M\contract e}(U) \geq \lambda_{M\contract e}(U\cap A) + \lambda_{M\contract e}(U\cup A).
    \end{align}
    Since $U\cup A$ is $U-T$-separating, we have $\lambda_{M\contract e}(U\cup A) \geq k$. Hence $\lambda_{M\contract e}(U\cap A) \leq k-1$. But $\lambda_{M}(U\cap A) = k$ since $U\cap A$ is $S-T$-separating. It follows that $e\in\closure_M(U\cap A)$, and in particular $e\in\closure_M(U)$. By Lemma \ref{lem:contractclosure}, we cannot have $e\in\closure_M(E(M)\smin (U\cup e))$. But then $\lambda_M(U\cup e) = k-1$, contradicting the fact that $\kappa_M(U,T) = k$.
\end{proof}

\begin{proof}[Proof of Theorem \ref{thm:mengerfragile}]
    We prove the result by induction on $|F|$, the case $|F| = 0$ being trivial. Suppose the result fails for a matroid $M$ with subsets $S, T, F$ as in the theorem. Let $k := \kappa_M(S,T)$ and $t := |F|$. For each $e \in F$, let $(A_e, B_e)$ be $S-T$-separating of order $k+1$ with $e \in A_e$ and $|A_e|$ as small as possible. Let $f$ be such that $|A_f| \leq |A_e|$ for all $e \in F$.
    
    \begin{claim}
        $A_f\cap F = \{f\}$.
    \end{claim}
    \begin{subproof}
        Suppose $g \in A_f\cap F$ with $g\neq f$. By our choice of $f$, we must have that $A_g = A_f$ (using Lemma \ref{lem:modular}). 
        Since $g$ is not flexible, Lemma \ref{lem:minimalsep} implies that $(A_f\smin g, B_f\cup g)$ is $S-T$-separating of order $k+1$, contradicting minimality of $|A_f|$.
    \end{subproof}
  By Lemma \ref{lem:growS} we can apply the theorem inductively, replacing $S$ by $A_f$ and $F$ by $F\smin f$, thus finding a sequence $(A_2, \ldots, A_t)$ of nested $A_f-T$-separating sets of order $k+1$. But now the sequence $(A_f, A_2, \ldots, A_t)$ satisfies all conditions of the theorem.
\end{proof}

We will use the following two facts:

\begin{lemma}\label{lem:contractinthemiddle}
    Let $M$ be a matroid, let $S,T$ be disjoint subsets of $E(M)$, let $k := \kappa_M(S,T)$, and let $(A_1, \ldots, A_t)$ be a sequence of nested $S-T$-separating sets of order $k+1$. 
    Let $(C,D)$ be a partition of $E(M)\smin (S\cup T)$ such that $C$ is independent, $D$ is coindependent, and $\lambda_{M\contract C \delete D}(S) = k$. Let $i, j \in \{1, \ldots, t\}$ with $i < j$. Let $C' := C \cap (A_j\smin A_i)$, let $D' := D\cap (A_j\smin A_i)$, and let $M' := M\contract C' \delete D'$. Then $(A_i, B_j)$ is $S-T$-separating of order $k+1$ in $M'$. Moreover, $M'|A_i = M|A_i$ and $M'|B_j = M|B_j$.
\end{lemma}

\begin{proof}
    Let $M' := M\contract C' \delete D'$. By definition of $C$ and $D$, $\kappa_{M'}(S,T) = k$. By monotonicity of $\lambda$, $\lambda_{M'}(A_i) = k$. It follows from Lemma \ref{lem:contractclosure} that for all $e \in C'$, $e \notin \closure_M (A_i \cup (C' \smin \{e\})$ and $e \notin \closure_M (B_j \cup (C' \smin \{e\})$. From this the second claim follows.
\end{proof}

\begin{lemma}[{Geelen, Gerards, and Whittle \cite[Lemma 4.7]{GGW07}}]\label{lem:indepsubset}
    Let $M$ be a matroid, let $S, T$ be disjoint subsets of $E(M)$, and let $k := \kappa_M(S,T)$. There exist sets $S_1\subseteq S$ and $T_1\subseteq T$ such that $|S_1| = |T_1| = \kappa_M(S_1,T_1) = k$.
\end{lemma}

\section{Proof of the result for finite fields}\label{sec:finitefield}
Let $M$ be a rank-$r$ matroid on ground set $E$. Write $M = M[D]$ if the $r\times E$ matrix $D$ (over field $\field$) represents $M$. For $S\subseteq E$, denote by $D[S]$ the submatrix of $D$ induced by the columns labeled by $S$, and denote by $\langle D[S]\rangle$ the vector space spanned by the columns of $D[S]$. To clean up notation we will write $\langle S\rangle$ for $\langle D[S]\rangle$ if $D$ is clear from the context.

Recall that, if $(A,B)$ is such that $\lambda_{M[D]}(A) = k$, then $\langle A\rangle \cap \langle B \rangle$ is a $k$-dimensional subspace of $\field^r$. Assume $\field = \GF(q)$. Denote by $M^+_{(A,B)}$ the matroid obtained from $M$ by adding a copy of $\PG(k-1,q)$ to $M$, such that in the representation it is contained in $\langle A\rangle \cap \langle B \rangle$. Furthermore, $M^+_A := M^+_{(A,B)}\delete B$ and $M^+_B := M^+_{(A,B)}\delete A$. Now we can carry out row operations to get $M^+_{(A,B)} = M[D']$, with
\begin{align*}
    D' = \kbordermatrix{ & A &  & X & & B\\
                         &     & \vline & 0 & \vline & 0\\
                         \cline{4-4} \cline{6-6}
                         &  \phantom{XXX}D_1\phantom{XXX}   & \vline  & &   \vline  &  \\
                         &  &\vline & P & \vline &   \\
                         &     & \vline  & &   \vline  & \phantom{XXX}D_2\phantom{XXX}  \\
                         \cline{2-2} \cline{4-4}
                         & 0 & \vline & 0 &  \vline &},
\end{align*}
where $P$ is a $k\times X$ matrix representing $\PG(k-1,q)$ (with elements labeled by $X$). We remark that $M^+_{(A,B)}$ is the generalized parallel connection of $M^+_A$ and $M^+_B$ along $X$ (cf. \cite[Section 11.4]{ox2}). The following lemma follows easily from Lemma \ref{lem:contractinthemiddle}.

\begin{lemma}\label{lem:contractguts}
    Let $M$ be a $\GF(q)$-representable matroid, let $S$ and $T$ be disjoint subsets of $E(M)$ with $\kappa_M(S,T) = k$, and let $(A,B)$ be $S-T$-separating of order $k+1$. Let $(C,D)$ be a partition of $E(M)\smin (S\cup T)$ such that $\lambda_{M\contract C \delete D}(S) = k$. Then $(M^+_{(A,B)} \contract C \delete D)|X = M^+_{(A,B)}|X$.
\end{lemma}

We repeat the main result, filling in an explicit value for the constant:

\newcommand{\cfun}{\ensuremath{\textrm{TODO}(q, k, n)}}
\begin{theorem}\label{thm:intertwinewithminorrepeat}
    Let $q$ be a prime power, let $M$ be a $\GF(q)$-representable matroid on ground set $E$, let $N$ be a minor of $M$ on $n$ elements, let $S,T \subseteq E$, and let $k := \kappa_M(S, T)$. If $|E\smin (S\cup T)| > n + 2(n+1)q^{n^2}$, then there exists an element $e \in E$ such that at least one of the following holds:
    \begin{enumerate}
        \item $\kappa_{M\delete e}(S, T) = k$ and $N$ is a minor of $M\delete e$;
        \item $\kappa_{M\contract e}(S, T) = k$ and $N$ is a minor of $M\contract e$.
    \end{enumerate}
\end{theorem}

The proof is not hard, but unfortunately we could not avoid using rather involved notation. For that reason we give a rough sketch of the idea. Let $M$ be a counterexample. First we construct a long sequence $(A_1, B_1), \ldots, (A_t, B_t)$ of nested $S-T$-separating partitions of order $k+1$. For each $i$ we define the matroid $M_i$, obtained from $M^+_{B_i}$ by deleting or contracting the elements of $B_i\smin E(N)$ so that the minor $N$ is preserved. Since each $M_i$ will have the same number of elements, only a finite number of distinct represented matroids can arise. Since our matroid is sufficiently large it follows that, after suitably relabeling the new elements, $M_i = M_j$ for some $i < j$. This shows that the elements in $A_j\smin A_i$ can be removed in such a way that both $N$ and the $S-T$-connectivity are preserved, which contradicts our choice of $M$.

\begin{proof}
    Let $q, M, N, n, S, T$, and $k$ be as stated, and assume $|E\smin (S\cup T)| > n + 2(n+1)q^{n^2}$, yet no element can be removed keeping both the $S-T$-connectivity and the minor $N$. Let $(C, D)$ be a partition of $E\smin (S\cup T)$ such that $\lambda_{M\contract C \delete D}(S) = \kappa_M(S,T)$ and such that $C$ is independent and $D$ coindependent. Let $(C_N, D_N)$ be a partition of $E\smin E(N)$ such that $N = M\contract C_N \delete D_N$ and such that $C_N$ is independent and $D_N$ coindependent. By our assumption, $C\cap C_N = \emptyset$ and $D\cap D_N = \emptyset$.    
    
     Let $F := C\cup D\smin E(N)$, and let $t'' := |F|$. Then $t'' > 2(n+1)q^{n^2}$. By Theorem \ref{thm:mengerfragile}, there is a nested sequence $(A''_1, \ldots, A''_{t''})$ of $S-T$-separating sets of order $k+1$ such that $A''_{i} \subsetneq A''_{i+1}$ for $i\in\{1,\ldots, t''-1\}$. Let $(f_1'', \ldots, f_{t''}'')$ be the corresponding ordering of $F$.
     Consider the sequence $(A''_1\cap E(N), \ldots, A''_{t''}\cap E(N))$. This sequence contains at most $n+1$ different elements. It follows that $(A_1'', \ldots, A''_{t''})$ has a subsequence $(A'_1, \ldots, A'_{t'})$ such that $A'_i \cap E(N) = A'_j\cap E(N)$ for all $i,j \in \{1,\ldots, t'\}$, and such that $t' \geq t''/(n+1) > 2 q^{n^2}$. 
    
    Let $(f'_1, \ldots, f'_{t'})$ be the corresponding subsequence of $F$. Using duality if necessary we may assume that $|\{f'_1, \ldots, f'_{t'}\}\cap C| \geq |\{f'_1, \ldots, f'_{t'}\}\cap D|$. Let $(A_1, \ldots, A_t)$ be a subsequence of $(A'_1, \ldots, A'_{t'})$ such that $A_{i+1}\smin A_i$ contains an element of $C$ for all $i\in \{1,\ldots, t-1\}$, and such that $t \geq t'/2 > q^{n^2}$. For each $i \in \{1,\ldots,t\}$, define $B_i := E \smin A_i$.
    
    Let $H$ be an $r\times E$ matrix over $\GF(q)$ representing $M$. Let $s := (q^{k}-1)/(q-1)$. For each $i$, let $W_i := \langle A_i \rangle \cap \langle B_i \rangle$, and let $X_i := \{x^i_1, \ldots, x^i_s\}$ be a set of labels disjoint from $E$ and disjoint from $X_j$ for all $j \in \{1,\ldots,t\}\smin\{i\}$. Let the $k\times X_1$ matrix $P_1$ be an arbitrary representation of $\PG(k-1,q)$ having ground set $X_1$. 
    
    For each $i \in \{1,\ldots, t\}$, let $M^+_i$ be the matroid $M^+_{(A_i,B_i)}$ with the set $X$ relabeled by $X_i$. Moreover, we assume this labeling was chosen such that, in $(M^+_1)^+_i \contract C \delete D$, $x^i_j$ is parallel to $x^1_j$ for all $j \in \{1,\ldots,s\}$ (where $(M^+_1)^+_i$ is defined in the obvious way). This can be done because of Lemma \ref{lem:contractguts}.
    
    Now we define, for each $i$, a matroid $N_i$ as follows: first set $N_i' := (M^+_i \delete A_i)\contract (C_N\cap B_i) \delete (D_N\cap B_i)$. Now $N_i$ is obtained from $N_i'$ by relabeling $x^i_j$ by $x^1_j$. Let $H_i$ be the corresponding representation matrix. Note that, for $i, j \in \{1,\ldots, t\}$, $E(N_i) = E(N_j) \subseteq E(N)\cup X_1$. Hence $|E(N_i)| \leq n + s$. Since $X_i \subseteq \langle B_i \rangle$, we find that $\rank(N_i) \leq n$. Furthermore, for all $x \in X_1$, $H_i[x] = H_j[x]$. Hence there are at most $((q^n -1)/(q-1) + 1)^n \leq q^{n^2}$ distinct representation matrices $H_i$. Since $t > q^{n^2}$, there exist $i, j \in \{1,\ldots, t\}$ with $i < j$ such that $H_i = H_j$. But then
        $M\contract (B_i\cap C_N) \delete (B_i\cap D_N)$ is equal to
    \begin{align*}
        \big(M\contract ((A_j\smin A_i)\cap C) \delete ((A_j\smin A_i)\cap D)\big) \contract (B_j\cap C_N) \delete (B_j \cap D_N),
    \end{align*}
    using Lemmas \ref{lem:contractinthemiddle} and \ref{lem:contractguts}. In particular, since $(A_j\smin A_i)\cap C \neq \emptyset$, there exists an $e \in C$ such that $\kappa_{M\contract e}(S,T) = k$ and $M\contract e$ has $N$ as minor, a contradiction.
\end{proof}
For completeness we show that Conjecture \ref{con:mainres} follows from Theorem \ref{thm:intertwinewithminorrepeat} when $M$ is $\GF(q)$-representable.
\begin{proof}[Proof of Conjecture \ref{con:mainres} for $\GF(q)$-representable matroids]
    Let $n := |Q\cup R|$ and set $c := n + 2(n+1)q^{n^2}$. Let $(C, D)$ be a partition of $E\smin (Q\cup R)$ such that $\lambda_{M\contract C \delete D}(Q) = k$. By Theorem \ref{thm:tuttelinkminmax}, $C$ and $D$ exist. Now apply Theorem \ref{thm:intertwinewithminorrepeat} with $N = M\contract C \delete D$, $S$, and $T$. The result follows.
\end{proof}


\section{Intertwining two connectivities}\label{sec:mainres}

In this section we prove Conjecture \ref{con:mainres} for all representable matroids. The key property we need for our proof is that we can add a point to the intersection of two non-skew flats. Formally:

\begin{definition}\label{def:intersectionclosed}
A matroid $M$ has the \emph{intersection property} if for all flats $S, T \in E(M)$ such that $\localconn(S,T) > 0$, there exist a matroid $N$ and a non-loop element $e\in E(N)$ such that $N\delete e = M$, and $e \in \closure_N(S)\cap\closure_N(T)$.  In this case, we say that $N$ is a \emph{good extension} of $M$ (with respect to $S,T$).
    A class of matroids $\mathcal{M}$ is \emph{intersection-closed} if every $M \in \mathcal{M}$ has the intersection property, and $\mathcal{M}$ is closed under minors, duality, and good extensions.
\end{definition}

Note that the class of representable matroids is evidently intersection-closed.   The V\'amos matroid shows that not all matroids have the intersection property.  See \cite{Bon11} for more on matroids with the intersection property. 

The restriction we use is reminiscent of the double-circuit property from the min-max theorem for matroid matching (see \cite{DL87}). However, whereas the min-max theorem is false even for affine spaces, in our case the condition appears to be just an artifact of our proof. We remain hopeful that Conjecture \ref{con:mainres} can be proven without this condition. We will now state and prove the main result.

\begin{theorem}\label{thm:mainres2}
    There exists a function $c: \N^2 \to \N$ with the following property. Let $M$ be a matroid in an intersection-closed family, and let $Q, R, S, T, F \subseteq E(M)$ be sets of elements such that $Q\cap R = S\cap T = \emptyset$ and $F \subseteq E(M)\smin (Q\cup R\cup S \cup T)$. Let $k := \kappa_M(Q,R)$ and $l := \kappa_M(S,T)$. If $|F| \geq c(k,l)$, then there exists an element $e\in F$ such that one of the following holds:
    \begin{enumerate}
        \item $\kappa_{M\delete e}(Q, R) = k$ and $\kappa_{M\delete e}(S, T) = l$;
        \item $\kappa_{M\contract e}(Q, R) = k$ and $\kappa_{M\contract e}(S, T) = l$.
    \end{enumerate}
\end{theorem}    

\begin{proof}
    We prove that the result holds for $c(k,l) := 4^{k+l}$. We proceed by induction on $k+l$, noting that the base case where $k = 0$ or $l = 0$ is straightforward. Assume that the result holds for all $k', l'$ with $k'+l' < k+l$, but that $M, Q, R, S, T, F$ form a counterexample. 
    Possibly after relabeling we may assume $k \leq l$. By Lemma \ref{lem:indepsubset} we can assume that $|S| = |T| = l$, and that $S$ and $T$ are independent sets. Furthermore, we can assume that for each $e\in F$ either $\kappa_{M\delete e}(Q,R) < k$ or $\kappa_{M\contract e}(Q,R) < k$.
    
    \begin{claim}
        There exists a $Q-R$ separating partition $(A,B)$ with $\lambda(A)= k$, such that $A\cap S \neq \emptyset$, $A\cap T \neq \emptyset$, $|A\cap (S\cup T)| \geq l$, and $|B\cap F| \geq \lfloor\frac{1}{2} c(k,l)\rfloor$.
    \end{claim}
    \begin{subproof}
        Let $(A_1, \ldots, A_t)$ be the nested sequence of $Q-R$ separating sets from Theorem \ref{thm:mengerfragile}, let $(B_1, \ldots, B_t)$ be their complements, and let $(f_1, \ldots, f_t)$ be the corresponding ordering of $F$. Let $i := \lfloor t/2\rfloor$. First we show that one of $A_i$ and $B_i$ meets both of $S$ and $T$. Indeed: otherwise we have (possibly after swapping $S$ and $T$) that $S\subseteq A_i$ and $T\subseteq B_i$. In that case $(A_i, B_i)$ is $S-T$ separating with $\lambda_M(A_i) = k$. It follows that $k = l$. Assume $f_i$ is non-contractible with respect to $(Q,R)$. Then $\lambda_{M\contract f_i}(B_i) = k-1$, and therefore $f_i$ is also non-contractible with respect to $(S,T)$, so the theorem holds with $e = f_i$.
        
        Hence, possibly after exchanging the sequences $(A_1, \ldots, A_t)$ and $(B_1, \ldots, B_t)$, we can assume $A_i\cap S \neq \emptyset$ and $A_i\cap T \neq \emptyset$. If $|A_i\cap (S\cup T)| < l$ then $|B_i\cap (S\cup T)| > l$, and therefore $(A, B) = (B_i, A_i)$ is a partition as desired; otherwise we simply take $(A,B) = (A_i, B_i)$.
    \end{subproof}
    
    If necessary, we relabel $Q$ and $R$ so that $Q\subseteq A$ and $R\subseteq B$.
    Define
    \begin{align*}
     S_1 := A\cap S &\qquad T_1 := A\cap T\\
     S_2 := B\cap S &\qquad T_2 := B\cap T.  
    \end{align*}
     Also define $F_2 := B\cap F$. We try to remove the elements from $A$ while preserving the $S-T$ connectivity. Let $N_0 := M$, and order the elements of $A\smin (S_1\cup T_1)$ arbitrarily as $a_1, \ldots, a_u$. For $i = 1, 2, \ldots, u$ define $N_i$ as follows. If $\kappa_{N_{i-1}\delete a_i}(S,T) = l$ and $a_i \not \in \coclosure_{N_{i-1}}(B)$, then $N_i := N_{i-1}\contract a_i$. Else, if $\kappa_{N_{i-1}\contract a_i}(S,T) = l$ and $a_i \not \in \closure_{N_{i-1}}(B)$, then $N_i := N_{i-1}\contract a_i$. Otherwise $N_i := N_{i-1}$. Observe that $\kappa_{N_u}(S,T) = l$ and $\kappa_{N_u}(A\cap E(N_u), R) = \lambda_{N_u}(B) = k$. We distinguish two cases.
    
    \paragraph{Case I: $\localconn_{N_u}(S_1,T_1) > 0$.}
    Since $N_u$ is a member of an intersection-closed family, we can find a matroid $N^+$ in this family with a non-loop element $s$ such that $N^+\delete s = N_u$, and $s\in \closure_{N^+}(S_1)\cap\closure_{N^+}(T_1)$. We distinguish two subcases:
    
    \paragraph{Case Ia: $s \not\in \closure_{N^+}(B)$.}
    Let $N := N^+ \contract s$, and define $Q' := A\cap E(N)$. Then $\kappa_{N}(S,T) = l-1$ and $\kappa_N(Q', R) = k$. Since $|F_2| \geq c(k, l-1)$, by induction we can find an element $e \in F_2$ such that either $\kappa_{N\contract e}(S,T) = l-1$ and $\kappa_{N\contract e}(Q',R) = k$, or $\kappa_{N\delete e}(S,T) = l-1$ and $\kappa_{N\delete e}(Q',R) = k$. We assume the former, and remark that the proof for the latter case is similar.
    
    \begin{claim}
        $\kappa_{M\contract e}(Q,R) = k$ and $\kappa_{M\contract e}(S,T) = l$.
    \end{claim}
    \begin{subproof}
        Suppose $\kappa_{M\contract e}(Q,R) < k$, that is, $e$ is non-contractible with respect to $(Q,R)$. By Lemma \ref{lem:growS}, $e$ is also non-contractible with respect to $(A,R)$ in $M$. But $(A, B)$ is $Q'-R$ separating, so we must have $\lambda_{M\contract e}(A) = k$, a contradiction.       
        %
        
            Next, let $C, D$ be such that $C$ is independent in $N$, $e\in C$ and, in $N_0 := N\contract C \delete D$, we have $E(N_0) = S\cup T$ and $\lambda_{N_0}(S) = l-1$. Since $C$ is independent in $N^+\contract s$, it follows that $s$ is not a loop in $N^+\contract C$. Let $N_0^+ := N^+\contract C \delete D$. Since $s\in\closure_{N_0^+}(S) \cap \closure_{N_0^+}(T)$, we must have that $\lambda_{N_0^+\delete s}(S) = l$. It follows that $\kappa_{M\contract e}(S, T) = l$ as desired. 
    \end{subproof}
    
    \paragraph{Case Ib: $s \in \closure_{N^+}(B)$.} Again we define $Q' := A\cap E(N)$. Let $(A_1, \ldots, A_{t'})$ be the nested sequence of $Q'-R$ separating sets in $N^+$ from Theorem \ref{thm:mengerfragile} (applied to $N, Q', R$, and $F_2$), let $(B_1, \ldots, B_{t'})$ be their complements, and let $(f_1, \ldots, f_{t'})$ be the corresponding ordering of $F_2$. Let $j := c(k-1, l-1)$. If $s \not\in \closure_{N^+}(B_{j})$ then we apply the arguments from Case (Ia) with $A_{j}\cap E(N)$ replacing $Q'$, $B_{j}$ replacing $B$, and $F\cap B_{j}$ replacing $F_2$. Otherwise, let $N := N^+\contract s$, define $R' := B_{j}$ and $F_2' := F_2\smin B_{j}$. We have $\kappa_N(Q', R') = k-1$ and $\kappa_N(S,T) = l-1$. Since $|F_2'| \geq c(k-1,l-1)$, we find by induction an element $e\in F_2'$ such that either $\kappa_{N\contract e}(Q',R') = k-1$ and $\kappa_{N\contract e}(S,T) = l-1$, or $\kappa_{N\delete e}(Q',R') = k-1$ and $\kappa_{N\delete e}(S,T) = l-1$. We assume the latter, and remark that the proof in the former case is similar.

        \begin{claim}\label{cl:previousclaim}
            $\kappa_{M\delete e}(Q,R) = k$ and $\kappa_{M\delete e}(S,T) = l$.
        \end{claim}
        \begin{subproof}
            Suppose $e = f_i'\in F_2'$ is non-deletable with respect to $(Q,R)$. Then $e\in \coclosure_{N_u}(B_{i'})$, so $\lambda_{N_u\delete e}(B_{i'}) = k-1$. But $s \in \closure_{N^+}(B_{i'}) \cap \closure_{N^+}(A_{i'}\smin e)$, so we must have $\lambda_{N^+\delete e}(B_{i'}) = k-1$. But then $\lambda_{N^+\delete e \contract s}(B_{i'}) = k-2$, contradicting our choice of $e$. Hence $e$ is deletable with respect to $(Q,R)$.
            
            The proof that $\kappa_{M\delete e}(S,T) = l$ is the same as before and we omit it.
        \end{subproof}
        
        \paragraph{Case II: $\localconn_{N_u}(S_1,T_1) = \colocalconn_{N_u}(S_1,T_1) = 0$.} By dualizing if necessary, we may assume there is an element $e\in \closure_{N_u}(A)\cap \closure_{N_u}(B) \cap F$, i.e. an element that is deletable with respect to $(Q,R)$ in $M$. We assume $e\in A$ (replacing $(A,B)$ by $(A\cup e, B\smin e)$ otherwise).
        
        \begin{claim}
            $e\in\closure_{N_u}(S_1\cup T_1)$. 
        \end{claim}
        
        \begin{subproof}
            First we show that $\coclosure_{N_u}(B)\smin (S_1\cup T_1)$ spans $S_1\cup T_1$. Suppose not, and let $X := (S_1 \cup T_1) \smin \coclosure_{N_u}(B)$. By construction of $N_u$, all remaining elements are in $\closure_{N_u}(B)$, so we have that $N_u \delete X$ has lower rank than $N_u$. Hence $X$ contains a cocircuit. But this contradicts the fact that $S_1$ and $T_1$ are coskew.
            
            Now pick $B' := \coclosure_{N_u}(B)\smin (S_1\cup T_1 \cup e)$ and $A' := A\smin B'$. Then $k' := \lambda_{N_u}(A') \leq k$. But since $S_1\cup T_1 \cup e \subseteq A'$ and $S_1\cup T_1\cup e \subseteq \closure_{N_u}(B')$, we must have that $\rank_{N_u}(S_1 \cup T_1 \cup e) \leq k' \leq k \leq l$. Note that $|S_1\cup T_1| \geq l$ and, since $S_1$ and $T_1$ are skew, $\rank_{N_u}(S_1\cup T_1) \geq l$. It follows that $k' = k = l$, and therefore $e\in \closure_{N_u}(S_1\cup T_1)$ as desired.
        \end{subproof}
        Similar to before, we define $Q' := A\cap E(N_u)\smin \{e\}$. Let $(A_1, \ldots, A_{t'})$ be the nested sequence of $Q'-R$ separating sets in $N_u$ from Theorem \ref{thm:mengerfragile} (applied to $Q', R$, and $F_2$), let $(B_1, \ldots, B_{t'})$ be their complements, and let $(f_1, \ldots, f_{t'})$ be the corresponding ordering of $F_2$. Let $j := c(k-1, l)$. Again we distinguish two cases.

        \paragraph{Case IIa: $e\not\in\closure_{N_u}(B_j)$.} Let $N_{v}$ be obtained from $N_u$ by contracting $e$ and removing the other elements from $A_j$ according to the same rules used to obtain $N_u$. We can then apply the arguments of Case I to $N_v$ (with $A_j$ replacing $A$ and $B_j$ replacing $B$), observing that $|F\cap B_j| \geq 2(c(k-1,l-1)+ c(k,l-1))$.

        \paragraph{Case IIb: $e \in \closure_{N_u}(B_j)$.} Let $N := N_u\contract e$, and define $R' := B_j$. By induction we find an element $f \in \{f_1, \ldots, f_{j}\}$ such that either $\kappa_{N\contract f}(Q',R') = k-1$ and $\kappa_{N\contract f}(S,T) = l$, or $\kappa_{N\delete f}(Q',R') = k-1$ and $\kappa_{N\delete f}(S,T) = l$.  As before, in the former case we have 
        $\kappa_{M\contract f}(Q,R) = k$ and $\kappa_{M\contract f}(S,T) = l$ and in the latter case we have $\kappa_{M\delete f}(Q,R) = k$ and $\kappa_{M\delete
         f}(S,T) = l$.
     This completes the proof of the theorem.
\end{proof}


\paragraph{Acknowledgements}
We thank Jim Geelen for suggesting the problem to us, and for suggesting the proof approach of Theorem \ref{thm:mainres2}. We thank Bert Gerards for his support and several valuable insights.

\bibliography{matbib2012}

\begin{thebibliography}{10}

\bibitem{Bon11}
Joseph~E. Bonin.
\newblock A note on the sticky matroid conjecture.
\newblock {\em Ann. Comb.}, 15(4):619--624, 2011.

\bibitem{DL87}
A.~Dress and L.~Lov{\'a}sz.
\newblock On some combinatorial properties of algebraic matroids.
\newblock {\em Combinatorica}, 7(1):39--48, 1987.

\bibitem{GW10}
J.~Geelen and G.~Whittle.
\newblock Inequivalent representations of matroids over prime fields.
\newblock Submitted. Preprint at arXiv:1101.4683, 2011.

\bibitem{GGW02}
James~F. Geelen, A.~M.~H. Gerards, and Geoff Whittle.
\newblock Branch-width and well-quasi-ordering in matroids and graphs.
\newblock {\em J. Combin. Theory Ser. B}, 84(2):270--290, 2002.

\bibitem{GGW07}
Jim Geelen, Bert Gerards, and Geoff Whittle.
\newblock Excluding a planar graph from {${\rm GF}(q)$}-representable matroids.
\newblock {\em J. Combin. Theory Ser. B}, 97(6):971--998, 2007.

\bibitem{GHW05}
Jim Geelen, Petr Hlin{\v{e}}n{\'y}, and Geoff Whittle.
\newblock Bridging separations in matroids.
\newblock {\em SIAM J. Discrete Math.}, 18(3):638--646 (electronic), 2004/05.

\bibitem{Kral07}
Daniel Kr\'{a}l'.
\newblock Computing representations of matroids of bounded branch-width.
\newblock In {\em Proceedings of the 24th annual conference on Theoretical
  aspects of computer science}, STACS'07, pages 224--235, Berlin, Heidelberg,
  2007. Springer-Verlag.

\bibitem{Mar09}
D{\'a}niel Marx.
\newblock A parameterized view on matroid optimization problems.
\newblock {\em Theoret. Comput. Sci.}, 410(44):4471--4479, 2009.

\bibitem{ox2}
J.~Oxley.
\newblock {\em Matroid Theory, Second Edition}.
\newblock Oxford University Press, 2011.

\bibitem{OSW10a}
J.~Oxley, C.~Semple, and G.~Whittle.
\newblock Exposing 3-separations in 3-connected matroids.
\newblock {\em Adv. in Appl. Math.}
\newblock Accepted.

\bibitem{RSXX}
Neil Robertson and P.~D. Seymour.
\newblock Graph minors. {XX}. {W}agner's conjecture.
\newblock {\em J. Combin. Theory Ser. B}, 92(2):325--357, 2004.

\bibitem{RSXXI}
Neil Robertson and P.~D. Seymour.
\newblock Graph minors. {XXI}. graphs with unique linkages.
\newblock {\em J. Combin. Theory Ser. B}, 99(3):583--616, 2009.

\bibitem{TruIII}
K.~Truemper.
\newblock A decomposition theory for matroids. {III}. {D}ecomposition
  conditions.
\newblock {\em J. Combin. Theory Ser. B}, 41(3):275--305, 1986.

\bibitem{Tut65a}
W.~T. Tutte.
\newblock Menger's theorem for matroids.
\newblock {\em J. Res. Nat. Bur. Standards Sect. B}, 69B:49--53, 1965.

\end{thebibliography}
\bibliographystyle{plain}

\end{document}